\documentclass[11pt]{amsart}
\usepackage{amsmath,amsthm,amssymb,fullpage}

\def\R{\mathbb{R}}

\def\p{{\mathfrak p}}

\def\lcm{\mbox{lcm}}
\def\A{{\mathcal A}}
\def\Pp{{\mathcal P}}

\def\rad{\mathrm{rad}}

\newcommand{\leg}[2]{\genfrac{(}{)}{}{}{#1}{#2}}

\newtheorem{thm}{Theorem}[section]
\newtheorem{prop}[thm]{Proposition}
\newtheorem{lem}[thm]{Lemma}

\newtheorem{conj}[thm]{Conjecture}

\theoremstyle{remark}
\newtheorem*{rmk}{Remark}
\newtheorem*{examples}{An example and a non-example}

\begin{document}

\title{The reciprocal sum of  divisors of Mersenne numbers}
\author{Zebediah Engberg}
\address{Wasatch Academy\\ Mt. Pleasant, UT 84647}
\email{zeb.engberg@wasatchacademy.org}

\author{Paul Pollack}
\address{Department of Mathematics\\ University of Georgia\\ Athens, GA 30602}
\email{pollack@uga.edu}

\begin{abstract} We investigate various questions concerning the reciprocal sum of  divisors, or prime divisors, of the Mersenne numbers $2^n-1$. Conditional on the Elliott--Halberstam Conjecture and the Generalized Riemann Hypothesis, we determine $\max_{n\le x} \sum_{p \mid 2^n-1} 1/p$ to within $o(1)$ and $\max_{n\le x} \sum_{d\mid 2^n-1}1/d$ to within a factor of $1+o(1)$, as $x\to\infty$. This refines, conditionally, earlier estimates of Erd\H{o}s and Erd\H{o}s--Kiss--Pomerance. Conditionally (only) on GRH, we also determine $\sum 1/d$ to within a factor of $1+o(1)$ where $d$ runs over all numbers dividing $2^n-1$ for some $n\le x$. This conditionally confirms a conjecture of Pomerance and answers a question of Murty--Rosen--Silverman. Finally, we show that both $\sum_{p\mid 2^n-1} 1/p$ and $\sum_{d\mid 2^n-1}1/d$ admit continuous distribution functions in the sense of probabilistic number theory.
\end{abstract}

\maketitle

\section{Introduction}
Let $f(n) = \sum_{p\mid 2^n-1} 1/p$ and $F(n) = \sum_{d\mid 2^n-1}1/d$, so that $f(n)$ is the reciprocal sum of the primes dividing the $n$th Mersenne number, and $F(n)$ is the corresponding sum over all divisors. In this note we are concerned with the statistical distribution of $f(n)$ and $F(n)$ as $n$ varies, continuing earlier investigations of Erd\H{o}s \cite{erdos51, erdos71}, Pomerance \cite{pomerance86}, and Erd\H{o}s--Kiss--Pomerance \cite{EKP91}.

It will be convenient in what follows to write $\ell(d)$ for the multiplicative order of $2$ modulo the odd number $d$.

Our starting point is the following observation. Let $z\ge 2$, and let $n$ be the least common multiple of the natural numbers not exceeding $z$. Then
\[ f(n) = \sum_{p\mid 2^n-1}\frac{1}{p} = \sum_{\ell(p) \mid n} \frac{1}{p} \ge \sum_{2 < p \le z}\frac{1}{p},  \]
the last inequality holding on account of having $\ell(p) < p \le z$ for each odd prime $p\le z$. As is well-known \cite[Theorem 427, p. 466]{HW08}, there is a constant $C_1 = 0.26149\dots$ (the \textsf{Meissel--Mertens constant}) with the property that $\sum_{p \le z} \frac{1}{p} = \log\log{z} + C_1 + o(1)$ as $z\to\infty$. Hence,
\[ f(n) \ge \log\log{z} + C_1 - \frac{1}{2}+o(1). \]
Now taking $z=\frac{1}{2}\log{x}$, we deduce from the prime number theorem that
\begin{equation}\label{eq:easylower} \max_{n\le x} f(n) \ge \log\log\log{x} + C_1 -\frac{1}{2}+o(1). \end{equation}

Perhaps surprisingly, given the ease with which \eqref{eq:easylower} was established, this lower bound on $\max_{n\le x} f(n)$ is close to best possible. Indeed, in 1971 Erd\H{o}s proved the complementary upper bound
\begin{equation}\label{eq:erdosfupper} \max_{n\le x} f(n) \le \log\log\log{x} + O(1). \end{equation}
In the same paper \cite{erdos71}  he asks whether equality holds in \eqref{eq:easylower}.

Erd\H{o}s's question seems to have lain dormant for 20 years. In 1991, the study of $f(n)$ was resurrected by Erd\H{o}s, Kiss, and Pomerance \cite{EKP91}. The main thought there is to study the local distribution of $f(n)$. However, on p.\ 271 of their paper they briefly digress to give a negative answer to Erd\H{o}s's question, by showing that
\begin{equation*} \max_{n\le x} f(n) \ge \log\log\log{x} + C \end{equation*}
for some $C > C_1-\frac{1}{2}$ and all large $x$. In the PhD thesis of the first author \cite{engberg14}, it is shown using lower bounds for counts of smooth shifted primes (cf. the discussion in \S\ref{sec:sharpmaxprelim} below) that $C$ can be taken as $C_1-\frac{1}{2}+ 0.2069$.

Our first theorem determines the  ``correct'' value of $C$, conditional on the Generalized Riemann Hypothesis (GRH) and (a weak form of) the Elliott--Halberstam Conjecture (EHC).\footnote{By GRH, we always mean the Riemann Hypothesis for Dedekind zeta functions. The relevant version of the Elliott--Halberstam Conjecture is stated precisely in \S\ref{sec:sharpmaxprelim}.}

\begin{thm}[conditional on GRH and weak EHC]\label{thm:main1}  As $x\to\infty$,
\[ \max_{n\le x} f(n) = \log\log\log{x} + C_1 -\frac{1}{2} + C_2 + o(1). \]
Here $C_1$ is the Meissel--Mertens constant and $C_2 = \int_{1}^{\infty} \rho(u) u^{-1}\, \mathrm{d}u$, where $\rho(u)$ is Dickman's function {\rm (}see \S\ref{sec:sharpmaxprelim}{\rm )}.
\end{thm}
\noindent One can compute that $C_2= 0.521908\dots$ and that $C_1 -\frac{1}{2}+C_2 = 0.283405\dots$.

Since $\sum_{d\mid m}\frac{1}{d} \asymp \exp(\sum_{p\mid m}\frac{1}{p})$ for all positive integers $m$, the estimate \eqref{eq:erdosfupper} is equivalent to the bound
\[ \max_{n \le x} F(n) \ll \log\log{x}, \]
and this is how Erd\H{o}s advertises the result in \cite{erdos71}. Our proof of Theorem \ref{thm:main1} is easily adapted to determine, under the same hypotheses as Theorem \ref{thm:main1}, $\max_{n\le x} F(n)$ to within a factor of $1+o(1)$.

\begin{thm}[conditional on GRH and weak EHC]\label{thm:main1point5}  As $x\to\infty$,
\[ \max_{n\le x} F(n) = \exp(\gamma -\log{2} + C_2 + o(1)) \log\log{x}. \]
Here $\gamma$ is the familiar Euler--Mascheroni constant.
\end{thm}

Observe that $F(n) = \sum_{m \mid n} E(m)$, where $E(m) = \sum_{\ell(d) = m} 1/d$. Erd\H{o}s showed in \cite{erdos51} (see that paper's eq.\ (8)) that $\sum_{m\le x} E(m) \ll \log{x}$, as a part of a simple (re)proof of Romanov's theorem asserting that $\sum_{d\text{ odd}} 1/d\ell(d) < \infty$. Several questions on the average and extremal orders of $E(m)$ are left open in \cite{erdos51}; many of these were later taken up by Pomerance in \cite{pomerance86}. Our concern here is with the average order. Clearly,
\begin{equation}\label{eq:trivEmlower} \sum_{m \le x} E(m) = \sum_{\ell(d)\le x} \frac{1}{d} \ge \sum_{\substack{d \text{ odd} \\ d\le x}} \frac{1}{d} > \frac{1}{2} \log{x}, \end{equation}
and it was conjectured in \cite{pomerance86} that this is asymptotically sharp. Our next theorem shows that this conjecture follows from GRH.

\begin{thm}[conditional on GRH]\label{thm:1andthreequarters} As $x\to\infty$,
\[ \sum_{m \le x} E(m) = \left(\frac{1}{2}+o(1)\right) \log{x}. \]
\end{thm}

\noindent As we remark after its proof, Theorem \ref{thm:1andthreequarters} allows us to (conditionally) answer a question of Murty, Rosen, and Silverman concerning the behavior of $\sum_{d\text{ odd}} 1/d\ell(d)^{\epsilon}$ as $\epsilon \downarrow 0$.

Our final result belongs to the domain of probabilistic number theory. Recall that a function $D\colon \R\to\R$ is a \textsf{distribution function} if $D$ is nondecreasing, right-continuous, and obeys the boundary conditions
\[ \lim_{c\to-\infty} D(c) =0 \quad\text{and}\quad \lim_{c\to\infty} D(c) = 1. \]
If $g(n)$ is a real-valued arithmetic function, we say that $g(n)$ \textsf{has a distribution function} (or \textsf{possesses a limit law}) if there is a distribution function $D$ with the following property: For every real number $c$ at which $D$ is continuous, the set of natural numbers $n$ with $g(n) \le c$ has asymptotic density $D(c)$. It is easy to show that if $g(n)$ has a distribution function $D$, then $D$ is uniquely determined. If this $D$ is continuous everywhere, we say that \textsf{$g$ has a continuous distribution function}.

On p.\ 270 of \cite{EKP91}, Erd\H{o}s, Kiss, and Pomerance claim that it can be proved by ``more or less standard methods'' that $f(n)$ possesses a continuous distribution function. That $f(n)$ and $F(n)$ have distribution functions follows immediately from (later) general results of Luca and Shparlinski \cite{LS07}, who used the method of moments. (See the proof of Theorem 3(2) in \cite{luca05} for an alternative approach to these results.) However, we are not aware of any proof in the literature --- before \cite{EKP91} or since --- that the distribution function of $f$ is continuous.\footnote{Related distribution functions are claimed to be continuous in \cite[Theorem 3(2)]{luca05}. However, the offered proofs only show existence of the distribution functions, leaving continuity unaddressed.}  As the details of the proof seem to us both interesting and nontrivial, we have made our second goal with this note to fill this (much-unneeded) gap in the literature.

\begin{thm}\label{thm:main2} The distribution function of $f(n)$ is continuous. The same holds for $F(n)$.
\end{thm}


The reader interested in these problems may also wish to refer to \cite{murata04}, which describes another application of the GRH to the study of prime divisors of Mersenne numbers.

\section{The sharp maximal order of $f(n)$}\label{sec:sharpmax}
\subsection{Preliminaries on smooth shifted primes}\label{sec:sharpmaxprelim} Our proofs rely heavily on the notion of ``smoothness''. Let $P(n)$ denote the largest prime factor of $n$, with the convention that $P(1)=1$. The natural number $n$ is called \textsf{$Y$-smooth} (or \textsf{$Y$-friable}) if $P(n)\le Y$. For $X,Y\ge 1$, we let
\[ \Psi(X,Y) = \#\{n\le X: P(n)\le Y\}. \]
After work of de Bruijn, Hildebrand, Tenenbaum, and others, we have precise estimates for $\Psi(X,Y)$ in a wide range of $X$ and $Y$; a useful survey is \cite{granville08}. A fundamental role in this area is played by the \textsf{Dickman function} $\rho(u)$, defined as the solution to the delay differential equation
\[ u \rho'(u) + \rho(u-1) = 0 \]
with the initial condition  $\rho(u)=1$ for $0 \le u\le 1$. Hildebrand, extending earlier work of de Bruijn, showed that with $u:=\frac{\log{X}}{\log{Y}}$,
$$\Psi(X,Y) \sim X\rho(u), \qquad \text{as $X\to\infty$ with $X\ge Y > \exp((\log\log{X})^{1.7})$.}$$
(In fact, $1.7$ may be replaced with any constant exceeding $5/3$.) Given the close connection between $\Psi(x,y)$ and $\rho(u)$ it is important to have accurate estimates of $\rho(u)$. While much finer results are known, the crude estimate
\[ \rho(u) = 1/u^{u+o(u)}, \qquad\text{as $u\to\infty$}, \]
will suffice for all of our purposes.

In the arguments below, it is important to understand the frequency with which shifted primes $p-1$ are smooth. For $X, Y \ge 1$, set
\[ \Pi(X, Y) = \#\{p\le X: P(p-1) \le Y\}. \]
It is natural to expect that shifted primes are smooth with the same relative frequency as generic numbers of the same size. This heuristic, coupled with the prime number theorem, suggests that
\begin{equation}\label{eq:pomconjecture} \Pi(X,Y) \sim \frac{\Psi(X,Y)}{\log{X}} \end{equation}
in a wide range of $X$ and $Y$\!. An explicit conjecture of this kind appears in \cite{pomerance80}, where Pomerance proposes that \eqref{eq:pomconjecture} holds whenever $X,Y\to\infty$ with $X\ge Y$.

We require two results which represent partial progress towards Pomerance's conjecture. The first is conditional on the following weak form of the Elliott--Halberstam Conjecture.
\begin{conj}[weak Elliott--Halberstam Conjecture] For any fixed $\epsilon>0$, we have
$$\sum_{n\le x^{1-\epsilon}} \max_{\gcd(a,n)=1}\left|\pi(x;n,a)-\frac{\pi(x)}{\phi(n)}\right| = o(\pi(x)), \qquad \text{as  $x\to\infty$}.$$
\end{conj}

\begin{prop}[conditional on weak EHC]\label{prop:granville} Fix a real number $U > 1$. As $X\to\infty$,
\[ \Pi(X,Y) \sim \rho(u) \frac{X}{\log{X}}, \]
uniformly in the range $1 \le u \le U$.
\end{prop}

\noindent Proposition \ref{prop:granville}, which is due to Granville (unpublished), conditionally confirms \eqref{eq:pomconjecture} when $X\to\infty$ in the range $X\ge Y\ge X^{1/U}$.  A generalization of Proposition \ref{prop:granville} appears (with a full proof) as \cite[Theorem 1]{lamzouri07}. See also Lemma 4.1 in \cite{wang18}. We note that the proofs of Proposition \ref{prop:granville} use only the even weaker form of EHC where $\pi(x;n,a)$ is replaced with $\pi(x;n,1)$ and the max on $a$ is removed, and that  EHC enters the proofs of Theorems \ref{thm:main1} and \ref{thm:main1point5} only via Proposition \ref{prop:granville}.

The next result, due to Pomerance and Shparlinski (see \cite[Theorem 1]{PS02}), is unconditional and applies in a much wider range of $X$ and $Y$. But it is only an upper bound, not an asymptotic formula, and the bound itself is slightly weaker than the prediction \eqref{eq:pomconjecture}.

\begin{prop}\label{prop:PS} If $X$ is sufficiently large, and $X\ge Y \ge \exp(\sqrt{\log X\log\log{X}})$, then
\[ \Pi(X,Y) \ll u\rho(u) \frac{X}{\log{X}}. \]
\end{prop}

\subsection{Proof of the lower bound in Theorem \ref{thm:main1}}

\begin{lem}\label{lem:afterz} As $z\to\infty$, $$\displaystyle\sum_{\substack{p > z \\ P(p-1) \le z}} \frac{1}{p} \to C_2,$$ where $C_2 = \int_{1}^{\infty} \rho(u) u^{-1}\,\mathrm{d}u$ {\rm(}as in Theorem \ref{thm:main1}{\rm )}.
\end{lem}
\begin{proof} Fix a real number $U>1$. Using $\sum$ for the sum appearing in the lemma statement, we rewrite $\sum=\sum_1 + \sum_2+\sum_3$,  partitioning the range of summation as follows;
\[ \sum\nolimits_1\!:~z < p \le z^{U},\qquad \sum\nolimits_2\!:~z^U < p \le \exp((\log{z})^{3/2}), \qquad \sum\nolimits_3\!:~p > \exp((\log{z})^{3/2}).
\]

We estimate $\sum_1$ using Proposition \ref{prop:granville}. That result implies that, as $z\to\infty$,
\begin{align*}
 \sum\nolimits_1 &= \int_{z}^{z^U} \frac{\mathrm{d}\Pi(t,z)}{t} = \int_{z}^{z^U}\frac{\Pi(t,z)}{t^2}\,\mathrm{d}t+ o(1) \\
 &= (1+o(1)) \int_{z}^{z^U} \frac{1}{t\log{t}} \rho\left(\frac{\log{t}}{\log{z}}\right)\, \mathrm{d}t + o(1) = \int_{1}^{U} \rho(u) u^{-1}\,\mathrm{d}u + o(1),
\end{align*}
where in the last step we made the change of variables $t=z^u$. Hence,
\[ \lim_{z\to\infty} \sum\nolimits_1 = \int_{1}^{U}\rho(u) u^{-1}\,\mathrm{d}u.\]

A similar calculation, using Proposition \ref{prop:PS} in place of Proposition \ref{prop:granville}, reveals that
\[ \limsup_{z\to\infty} \sum\nolimits_2 \ll \int_{U}^{(\log{z})^{1/2}} \rho(u)\,\mathrm{d}u \le \int_{U}^{\infty} \rho(u)\,\mathrm{d}u, \]
where the constant implied by ``$\ll$'' is absolute (independent of $U$).

We handle $\sum_3$ using the trivial bound $\Pi(X,Y)\le \Psi(X,Y)$, along with the uniform estimate $\Psi(X,Y) \ll X \exp(-u/2)$, which holds for all $X\ge Y\ge 2$ (see \cite[Theorem 5.1, p. 512]{tenenbaum15}). We find that for large $z$,
\begin{align*} \sum\nolimits_3 &\le \int_{\exp((\log{z})^{3/2})}^{\infty} \frac{\Pi(t,z)}{t^2}\,\mathrm{d}{t} \le \int_{\exp((\log{z})^{3/2})}^{\infty} \frac{\Psi(t,z)}{t^2}\,\mathrm{d}{t}
\ll \int_{\exp((\log{z})^{3/2})}^{\infty} \frac{1}{t} \exp\left(-\frac{1}{2}\frac{\log{t}}{\log{z}}\right)\,\mathrm{d}t \\&= \log{z} \int_{(\log{z})^{1/2}}^{\infty} \exp\left(-\frac{1}{2}u\right)\,\mathrm{d}u \ll \log{z} \cdot \exp\left(-\frac{1}{2}(\log{z})^{1/2}\right).
\end{align*}
Thus,
$$ \lim_{z\to\infty} \sum\nolimits_3 = 0. $$

Since $\sum=\sum_1+\sum_2+\sum_3$, we deduce that
\begin{equation}\label{eq:limsupest} \limsup_{z\to\infty} \sum_{\substack{p > z \\ P(p-1) \le z}} \frac{1}{p} \le \int_{1}^{U}\rho(u) u^{-1} \, \mathrm{d}u + O\left(\int_{U}^{\infty} \rho(u)\, \mathrm{d}u\right), \end{equation}
and, since $\sum \ge \sum_1$, that
\begin{equation}\label{eq:liminfest} \liminf_{z\to\infty} \sum_{\substack{p > z \\ P(p-1) \le z}} \frac{1}{p} \ge \int_{1}^{U} \rho(u) u^{-1}\, \mathrm{d}u. \end{equation}
The lemma follows from \eqref{eq:limsupest} and \eqref{eq:liminfest} upon letting $U$ tend to infinity.
\end{proof}

\begin{proof}[Proof of the lower bound in Theorem \ref{thm:main1}] Let $x$ be large, let $z = \frac{1}{2}\log{x},$ and let $n$ be the least common multiple of the natural numbers  not exceeding $z$. By the prime number theorem, $\log{n} = (\frac{1}{2}+o(1))\log{x}$ as $x\to\infty$, and so in particular $n \le x$ for large $x$. Thus, for the lower bound in Theorem \ref{thm:main1}, it is enough to prove that as $x\to\infty$,
\begin{equation*} f(n) \ge \log\log\log{x} + C_1-\frac{1}{2} + C_2 + o(1). \end{equation*}

Observe that
\begin{align*} f(n) = \sum_{\ell(p)\mid n}\frac{1}{p} \ge \sum_{\substack{p>2 \\ p-1 \mid n}}\frac{1}{p} &= \sum_{2 < p \le z} \frac{1}{p} + \sum_{\substack{p > z \\ p-1 \mid n}} \frac{1}{p} \\ &= \sum_{2 < p \le z} \frac{1}{p} + \sum_{\substack{p > z \\ P(p-1) \le z}} \frac{1}{p} - E, \end{align*}
where
\[ E:= \sum_{\substack{p > z \\ P(p-1) \le z \\ p-1\nmid n}} \frac{1}{p}. \]
We will show shortly that $E=o(1)$ as $x\to\infty$. Assuming this for now, Mertens' theorem and Lemma \ref{lem:afterz} yield
\begin{align*}
f(n) &\ge (\log\log{z} + C_1 - 1/2) + C_2 + o(1) \\
&= \log\log\log{x} + C_1 - \frac{1}{2} + C_2 + o(1),
\end{align*}
as desired.

We turn now to proving that $E=o(1)$, as $z\to\infty$. In view of the upper bounds for $\sum_2$ and $\sum_3$ established in the proof of Lemma \ref{lem:afterz}, it will suffice to show that
\[ \sum_{\substack{z < p \le z^A \\ P(p-1) \le z \\ p-1\nmid n}} \frac{1}{p} = o(1), \]
for each fixed $A>1$. So suppose that $P(p-1)\le z$ and $p-1\nmid n$. Choose a prime power $\ell^e$ dividing $p-1$ with $\ell \le z$ and $\ell^e > z$. Then $e>1$, and $\ell^e$ is a squarefull divisor of $p-1$ exceeding $z$. It follows that
\begin{align*}
\sum_{\substack{z < p \le z^A \\ P(p-1) \le z \\ p-1\nmid n}} \frac{1}{p} &\le \sum_{\substack{m\text{ squarefull} \\ m > z}} \sum_{\substack{p \le z^A \\ m \mid p-1}} \frac{1}{p } \le \sum_{\substack{m\text{ squarefull} \\ m > z}} \sum_{\substack{p \le z^A \\ m \mid p-1}} \frac{1}{p-1} \\ &\le \sum_{\substack{m\text{ squarefull} \\ m > z}} \sum_{k \le z^A} \frac{1}{mk}\le \log(z^A+1) \sum_{\substack{m\text{ squarefull} \\ m > z}} \frac{1}{m} \ll \log(z^A+1) z^{-1/2}.
\end{align*}
The final expression tends to $0$ as $z$ tends to infinity.
\end{proof}

\subsection{Proof of the upper bound in Theorem \ref{thm:main1}} The Generalized Riemann Hypothesis enters into our proof by way of the following lemma of Kurlberg and Pomerance \cite[Theorem 23]{KP05}.

\begin{lem}[conditional on GRH]\label{lem:GRHlem} For all real $X, Y$ with $1\le Y \le \log{X}$,
\[ \#\left\{p\le X: \ell(p) \le \frac{p}{Y}\right\} \ll \frac{X}{Y\log X} + \frac{X\log\log{X}}{(\log{X})^2}. \]
\end{lem}

Call the odd prime $p$ \textsf{normal} if $\ell(p) > p/(\log\log{p})^2$ and \textsf{abnormal} otherwise. By Lemma \ref{lem:GRHlem}, the number of abnormal $p \in [X,2X]$ is $\ll \frac{X}{\log{X} \cdot (\log\log{X})^2}$ for all large $X$. Breaking $[3,X]$ into dyadic blocks, we deduce that the same bound holds for the number of abnormal $p\le X$. Hence,
\[ \sum_{p \text{ abnormal}}\frac{1}{p} < \infty, \]
which we will make use of momentarily.

\begin{proof}[Proof of the upper bound in Theorem \ref{thm:main1}] Let $x$ be large, let $n \le x$, and let $L$ be the least common multiple of the natural numbers not exceeding $(\log\log{x})^2$. If $p$ is a normal odd prime and $\ell(p)$ divides $n$, then
\[ p-1 = \ell(p) \cdot \frac{p-1}{\ell(p)} \mid n\cdot L. \]
We also have trivially that $p-1 \mid nL$ for all $p \le (\log\log{x})^2$. Thus,
\[ \sum_{\ell(p) \mid n}\frac{1}{p} \le \sum_{\substack{p > 2 \\ p-1\mid nL}} \frac{1}{p} + \sum_{\substack{p \text{ abnormal} \\ p > (\log\log{x})^2}} \frac{1}{p} \le \sum_{\substack{p > 2 \\ p-1\mid nL}} \frac{1}{p} +o(1). \]
(In the last step, we use that $\sum_{p\text{ abnormal}}\frac{1}{p} < \infty$.)  Now let $$z=\log{x} \cdot \log\log{x}.$$ Clearly,
\begin{equation}\label{eq:clearsum} \sum_{\substack{p > 2 \\ p-1\mid nL}} \frac{1}{p} \le \sum_{\substack{p > 2 \\ P(p-1) \le z}} \frac{1}{p} + \sum_{\substack{p-1 \mid nL \\ P(p-1) > z}}\frac{1}{p}.\end{equation}
The first right-hand sum satisfies
\begin{align*} \sum_{\substack{p > 2 \\ P(p-1) \le z}} \frac{1}{p} &= \sum_{2 < p \le z} \frac{1}{p} + \sum_{\substack{p > z \\ P(p-1) \le z}}\frac{1}{p} \\&= \log\log z + C_1 - \frac{1}{2} + C_2 + o(1) \\
&= \log\log\log x + C_1 - \frac{1}{2} + C_2 + o(1),
\end{align*}
and so to finish the proof it is enough to show that the second sum on the right of \eqref{eq:clearsum} is $o(1)$.

Suppose that $P(p-1)> z$ and $p-1\mid nL$. Then there is a prime $q>z$ dividing both $n$ and $p-1$. Hence,
\[ \sum_{\substack{p-1 \mid nL \\ P(p-1) > z}} \frac{1}{p} \le \sum_{\substack{q\mid n \\ q > z}} \sum_{\substack{d \mid nL \\ q \mid d}} \frac{1}{d} \le \sum_{\substack{q\mid n\\ q > z}}\frac{1}{q}\sum_{e\mid nL} \frac{1}{e} = \frac{\sigma(nL)}{nL} \sum_{\substack{q \mid n \\ q > z}}\frac{1}{q} \le \frac{\sigma(nL)}{nL} \cdot \frac{\omega(n)}{z}.     \]
We finish the proof by appealing to the known maximal orders of $\sigma$ and $\omega$.
By Theorem 323 on p.\ 350 of \cite{HW08}, $\frac{\sigma(nL)}{nL} \ll \log\log(nL)$. By the prime number theorem, $L < \exp(2(\log\log{x})^2)$, so that $nL < x \exp(2(\log\log{x})^2)$ and  $\log\log(nL) \ll \log\log{x}$. Also (see p.\ 471 in \cite{HW08}), $\omega(n) \ll \log x/\log\log{x}$. Thus,
\[ \frac{\sigma(nL)}{nL} \cdot \frac{\omega(n)}{z} \ll \log\log{x} \cdot \frac{\log{x}/\log\log{x}}{\log{x} \cdot \log\log{x}} = \frac{1}{\log\log{x}}, \]
which indeed tends to $0$ as $x$ tends to infinity.
\end{proof}

\section{The sharp maximal order of $F(n)$: Proof of Theorem \ref{thm:main1point5}} The proof of Theorem \ref{thm:main1point5} is based on the following variant of Theorem \ref{thm:main1}, where $1/p$ is replaced by $\log \frac{p}{p-1}$. The required changes in the proof are straightforward, bearing in mind that $\log \frac{p}{p-1} = 1/p + O(1/p^2)$.

\begin{prop}[conditional on GRH and weak EHC]\label{prop:main1variant} As $x\to\infty$, we have
\begin{equation}\label{eq:main1variant} \sum_{\ell(p) \mid n} \log \frac{p}{p-1} \le \sum_{2 < p \le \frac{1}{2}\log{x}} \log\frac{p}{p-1} + C_2 + o(1) \end{equation}
uniformly for all $n\le x$. Moreover, equality holds in \eqref{eq:main1variant} when $n$ is taken as the least common multiple of the natural numbers not exceeding  $\frac{1}{2}\log{x}$.
\end{prop}

\begin{proof}[Proof of Theorem \ref{thm:main1point5}] If $\ell(d) \mid n$, then $\ell(p) \mid n$ for all primes $p$ dividing $d$. Hence, $F(n) = \sum_{\ell(d)\mid n}1/d \le \prod_{\ell(p)\mid n}\left(1+1/p+1/p^2+\dots\right)$, so that by Proposition \ref{prop:main1variant},
\begin{align*} \log F(n) &\le \sum_{\ell(p)\mid n}\log \frac{p}{p-1} \\ &\le \log\bigg(\prod_{2 < p \le \frac{1}{2}\log{x}}(1-1/p)^{-1}\bigg) + C_2 + o(1) \\
&= \log\log\log{x} + \gamma -\log{2} + C_2 + o(1),
\end{align*}
using Mertens' product theorem in the final step. Exponentiating gives the upper inequality in Theorem \ref{thm:main1point5}.

We now let $n$ be the least common multiple of the integers up to $\frac{1}{2}\log{x}$ and prove that the lower inequality holds.  Suppose that $d$ has the form $p_1^{e_1} \cdots p_k^{e_k} q_1 \dots q_l$,
where the $p_i$ are distinct primes in $(2,\log\log x]$, each $p_i^{e_i} \le \frac{1}{2}\log{x}$, and the $q_i$ are distinct primes exceeding $\log\log{x}$ where each $\ell(q_i) \mid n$. As $\ell(p_i^{e_i}) \mid (p_i-1) p_i^{e_i-1} \mid n$ for each $i$, we see that
\[ \ell(d) = \lcm[\ell(p_1^{e_1}),\dots,\ell(p_k^{e_k}), \ell(q_1),\dots,\ell(q_l)] \mid n. \]
Hence, with $p$ and $q$ running over primes and $e_p:= \lfloor \frac{\log(\frac12 \log x)}{\log p}\rfloor$,
\begin{align*} F(n) = \sum_{\ell(d) \mid n}\frac{1}{d} &\ge \prod_{2 < p \le \log\log{x}} \left(1+\frac{1}{p} + \frac{1}{p^2} + \dots + \frac{1}{p^{e_p}}\right) \prod_{\substack{q>\log\log{x} \\ \ell(q)\mid n}}\left(1+\frac{1}{q}\right)\\
&=\prod_{2 < p \le \log\log{x}} \frac{p}{p-1} \left(1-\frac{1}{p^{e_p+1}}\right) \prod_{\substack{q>\log\log{x} \\ \ell(q)\mid n}}\frac{q}{q-1} \left(1-\frac{1}{q^2}\right)\\
&\ge \prod_{\ell(p) \mid n} \frac{p}{p-1} \prod_{2 < p \le \log\log{x}} \left(1-\frac{1}{p^{e_p+1}}\right) \prod_{q > \log\log{x}} \left(1-\frac{1}{q^2}\right).
\end{align*}
Noting that $p^{e_p+1} > \frac{1}{2}\log{x}$,
\begin{align*} \prod_{2 < p \le \log\log{x}} \left(1-\frac{1}{p^{e_p+1}}\right) &\ge 1 - \sum_{2 < p \le \log\log{x}} \frac{1}{p^{e_p+1}} \ge 1-\frac{1}{\frac{1}{2}\log x} \sum_{2 < p \le \log\log{x}} 1 = 1+o(1).
\end{align*}
The final product on $q$ appearing above is also $1+o(1)$. Hence,
\[ F(n) \ge (1+o(1)) \prod_{\ell(p) \mid n}\frac{p}{p-1} = \exp\left(\sum_{\ell(p)\mid n}\log \frac{p}{p-1} + o(1)\right). \]
Using that equality holds in \eqref{eq:main1variant} for our choice of $n$ and applying Mertens' product theorem in the same manner as before, the desired lower bound follows.
\end{proof}

\section{The reciprocal sum of primitive divisors: Proof of Theorem \ref{thm:1andthreequarters}}
Let $\lambda(n)$ denote the universal exponent for the multiplicative group modulo $n$. The next proposition, due to Friedlander, Pomerance, and Shparlinski \cite[Theorem 5]{FPS01}, shows that $\lambda(n)$ is usually not  much smaller than $n$.

\begin{prop}\label{prop:FPS} For all large $X$, and all $\Delta \ge (\log\log{X})^3$, the inequality $\lambda(n) > n \exp(-\Delta)$ holds for all positive integers $n\le X$ with at most $X \exp(-0.69 (\Delta \log \Delta)^{1/3})$ exceptions.
\end{prop}

The following lower bound for $\ell(n)$, in terms of $\lambda(n)$ and the numbers $\ell(p)$ for $p$ dividing $n$, is due to Kurlberg and Rudnick \cite[see \S5.1]{KR01}. See also Lemma 5 of \cite{KP05}.

\begin{prop}\label{prop:kplambdaell} For all odd numbers $n$,
\[ \ell(n) \ge \frac{\lambda(n)}{n} \prod_{p\mid n}\ell(p). \]
\end{prop}

In \cite{KP05}, it is proved that $\ell(n)$ is usually a bit larger than $n^{1/2}$, for almost all $n$. Specifically, if $\epsilon(x)$ is any function tending to $0$ as $x\to\infty$, then $\ell(n) > n^{1/2+\epsilon(n)}$ for all $n$ on a set of asymptotic density $1$. We need a variant of this result where the exceptional set is suitably small. Such a result is available if we are willing to slightly weaken the required lower bound on $\ell(n)$.

\begin{lem}\label{lem:finiterecipsum} Fix $\delta \in (0,\frac12)$. For all large $X$, the inequality $\ell(n) > n^{\frac12-\delta}$ holds for all odd $n\le X$ with at most $O(X/(\log{X})^2)$ exceptions.
\end{lem}

\begin{proof} Clearly, we can restrict attention to  $$ n > X/(\log{X})^2. $$ Proposition \ref{prop:FPS} (with $\Delta = \frac{1}{4}\delta\log{x}$) licenses us to assume that
\[ \lambda(n) \ge n \cdot X^{-\frac{1}{4}\delta}.\]
Using $\rad(n)$ for the largest squarefree divisor of $n$, we can also assume that
\[ \frac{n}{\rad(n)} \le (\log{n})^4. \]
Indeed, if this inequality fails, let $u:=n/\rad(n)$. Then $u\cdot \rad(u)$ is a squarefull divisor of $n$ with $u \cdot \rad(u) > 2 (\log{n})^4 > (\log{X})^4$. But the number of $n\le X$ with a squarefull divisor exceeding $(\log{X})^4$ is $O(X/(\log{X})^2)$, which is acceptable for us.

To continue, call the odd prime $p$ \textsf{good} if $\ell(p) > p^{\frac{1}{2}-\frac{1}{2}\delta}$ and \textsf{bad} otherwise. Note that if $p$ is bad, and $p\le T$, then $p$ divides $2^m-1$ for some $m \le T^{\frac{1}{2}-\frac{1}{2}\delta}$. Since each Mersenne number $2^m-1$ has fewer than $m$ distinct prime factors, the number of bad primes $p\le T$ is at most
\[ \sum_{m \le T^{\frac{1}{2}-\frac{1}{2}\delta}} m \le T^{1-\delta},  \]
for every real number $T\ge 1$.

Next, we show that the product of the bad primes dividing $n$ is usually small. Observe that
\begin{align*} \sum_{n\le X} \bigg(\sum_{\substack{p \mid n \\ p\text{ bad}}} \log{p}\bigg)^2 &\le \sum_{p\le X\text{ bad}} (\log{p})^2 \sum_{\substack{n\le X \\p\mid n}} 1 + \sum_{\substack{p, q\le X\text{ bad} \\ p \ne q}} \log{p}\log{q} \sum_{\substack{n \le X \\ pq\mid n}}1 \\
&\le X\left(\sum_{p\text{ bad}} \frac{(\log{p})^2}{p} + \left(\sum_{p\text{ bad}}\frac{\log{p}}{p}\right)^2 \right) \ll X.
\end{align*}
We use in the final step that the sums over bad primes are convergent, which follows by partial summation from the results of the last paragraph. Hence, by Markov's inequality,
\begin{equation}\label{eq:smallradnbad} \sum_{\substack{p \mid n \\ p\text{ bad}}} \log{p} < \frac{1}{2}\delta \log{X} \end{equation}
for all but $O(X/(\log{X})^2)$ integers $n\le X$. (The implied constant depends on $\delta$.)

Now suppose that $n$ is an odd number not exceeding $X$ and that $n$ is not one of the $O(X/(\log{X})^2)$ exceptional integers above. Write $n_{\text{bad}}$ for the largest divisor of $n$ composed entirely of bad primes, and define $n_{\text{good}}$ analogously. Then
\[ \ell(n) \ge \frac{\lambda(n)}{n} \prod_{\substack{p\text{ good} \\ p \mid n}} \ell(p) \ge X^{-\frac{1}{4}\delta} \prod_{\substack{p\text{ good} \\ p \mid n}} p^{\frac{1}{2}-\frac{1}{2}\delta} = X^{-\frac{1}{4}\delta} \cdot \rad(n_{\text{good}})^{\frac{1}{2}-\frac{1}{2}\delta}.   \]
Since
\[ \frac{n_{\text{good}}}{\rad(n_{\text{good}})}, \frac{n_{\text{bad}}}{\rad(n_{\text{bad}})} \le \frac{n}{\rad(n)} \le (\log{X})^4,  \]
\eqref{eq:smallradnbad} yields
\[ n_{\text{bad}} \le (\log{X})^4 \cdot \rad(n_{\text{bad}}) \le (\log{X})^{4} X^{\frac{1}{2}\delta}, \]
so that
\[ \rad(n_{\text{good}}) \ge \frac{n_{\text{good}}}{(\log{X})^4} = \frac{n/n_{\text{bad}}}{(\log{X})^4} \ge n X^{-\frac{1}{2}\delta} (\log{X})^{-8} \ge X^{1-\frac{1}{2}\delta} (\log{X})^{-10}. \]
Here the final inequality is justified by our assumption that $n> X/(\log{X})^2$. Putting this back above, we find that
\begin{align*} \ell(n) &\ge X^{-\frac{1}{4}\delta} \cdot X^{(1-\frac{1}{2}\delta)(\frac{1}{2}-\frac{1}{2}\delta)} (\log{X})^{-5} \\
&\ge X^{\frac{1}{2}-\delta} \ge n^{\frac{1}{2}-\delta},
\end{align*}
thereby completing the proof.
\end{proof}

The final piece of preparation needed for the proof of Theorem \ref{thm:1andthreequarters} is a theorem of Kurlberg (see \cite[Theorem 1]{kurlberg03}).

\begin{prop}[conditional on GRH]\label{prop:kurlberg} Fix $\epsilon \in (0,1)$. Then $\ell(n) > n^{1-\epsilon}$ for all odd $n$ except those belonging to a set of asymptotic density $0$.
\end{prop}

\begin{proof}[Proof of Theorem \ref{thm:1andthreequarters}] In view of the representation $\sum_{m \le x} E(m) = \sum_{\ell(d) \le x} 1/d$, it is enough to prove that for each $\epsilon > 0$ and all large $x$,
\begin{equation}\label{eq:toprove} \sum_{\substack{d > x \\ \ell(d) \le x}}\frac{1}{d} < \epsilon \log{x}.  \end{equation}
Certainly
\[ \sum_{\substack{x< d < x^{1+\frac{1}{2}\epsilon} \\ \ell(d) \le x}}\frac{1}{d} \le \sum_{x < d < x^{1+\frac{1}{2}\epsilon}}\frac{1}{d} < \frac{2}{3}\epsilon \log{x}\]
for all large $x$. To handle the $d$ with $x^{1+\frac{1}{2}\epsilon} < d \le x^3$, we appeal to Proposition \ref{prop:kurlberg}. If $d$ is in this range, and $\ell(d)\le x$, then $\ell(d) < d^{\delta}$ with $\delta = (1+\frac{1}{2}\epsilon)^{-1}$. By Proposition \ref{prop:kurlberg}, the count of $d\le X$ with $\ell(d) < d^{\delta}$ is $o(X)$, as $X\to\infty$, and now partial summation yields
\[ \sum_{\substack{x^{1+\frac{1}{2}\epsilon} < d \le x^3 \\ \ell(d) \le x}}\frac{1}{d}\le \sum_{\substack{x^{1+\frac{1}{2}\epsilon} < d \le x^3 \\ \ell(d) < d^{\delta}}}\frac{1}{d} = o(\log{x}), \]
as $x\to\infty$. Finally, if $d > x^3$ and $\ell(d) \le x$, then $\ell(d)< d^{1/3}$. But the $d$ with $\ell(d) < d^{1/3}$ comprise a set with finite reciprocal sum, by Lemma \ref{lem:finiterecipsum} and partial summation. Hence,
\[ \sum_{\substack{d > x^3 \\ \ell(d) \le x}}\frac{1}{d} \le \sum_{\substack{d > x^3 \\ \ell(d) < d^{1/3}}}\frac{1}{d} = o(1),  \]
as $x\to\infty$. Combining what was shown in the last few displays yields \eqref{eq:toprove}.
\end{proof}

\begin{rmk} Fix an integer $a\ge 2$, and let $\ell_{a}(d)$ denote the order of $a$ modulo $d$, where $d$ is assumed coprime to $a$. On p.\ 376 of \cite{MRS96}, Murty, Rosen, and Silverman observe that
\[ \sum_{\substack{d\ge 1 \\ \gcd(a,d)=1}} \frac{1}{d \cdot \ell_a(d)^{\epsilon}} \ge \sum_{\substack{d\ge 1 \\ \gcd(a,d)=1}} \frac{1}{d^{1+\epsilon}} = \zeta(1+\epsilon) \cdot \prod_{p\mid a} \left(1-\frac{1}{p^{1+\epsilon}}\right),  \]
and thus $\liminf_{\epsilon \downarrow 0} \epsilon \sum 1/d\ell_a(d)^{\epsilon} \ge \phi(a)/a$. They remark that ``it would be interesting to compute the exact value [of the limit] \dots if the limit exists''. Under GRH, we can show that this limit is exactly $\phi(a)/a$. To see the connection with our work, let $S_a(x) = \sum_{\ell_a(d) \le x} 1/d$, and observe that
\[ \sum_{\substack{d\ge 1 \\ \gcd(a,d)=1}} \frac{1}{d \cdot \ell_a(d)^{\epsilon}} = \int_{1^{-}}^{\infty} t^{-\epsilon} \, \mathrm{d}S_a(t). \]
That the limit in question is $\phi(a)/a$ now follows in a straightforward way from the asymptotic relation $S_a(x) \sim \frac{\phi(a)}{a} \log{x}$, as $x\to\infty$.\footnote{Conversely,  if $\epsilon \sum 1/d\ell_a(d)^{\epsilon} \to K$ as $\epsilon\downarrow 0$, then $S_a(x) = (K+o(1))\log{x}$, as $x\to\infty$. This follows from a Tauberian theorem of Hardy and Littlewood for Dirichlet series with nonnegative coefficients \cite[Theorem 16]{HL14}.} When $a=2$, that relation is exactly the assertion of Theorem \ref{thm:1andthreequarters}, while the proof for general $a$ is entirely analogous.
\end{rmk}

\section{Continuity of distribution functions: Proof of Theorem \ref{thm:main2}}

The following result in arithmetic combinatorics is the principal theorem of Pomerance and Sark\H{o}zy's paper \cite{PS88}.

\begin{prop}\label{prop:PSa} There are absolute constants $c_1$ and $N_1$ such that, if $N$ is a positive integer with $N\ge N_1$, $\Pp$ is a set of prime numbers not exceeding $N$ with
\[ \sum_{p \in \Pp} \frac{1}{p} > c_1, \]
$\A \subset \{1,2,\dots,N\}$, and
\begin{equation*}
\sum_{a \in \A} \frac{1}{a} > 10 (\log{N}) \left(\sum_{p \in \Pp}\frac{1}{p}\right)^{-1/2}, \end{equation*}
then there exist integers $a < a'$ in $\mathcal{A}$ for which $a\mid a'$ and
\[ \frac{a'}{a} \mid \prod_{p \in \Pp} p. \]
\end{prop}

We distill what we need from Proposition \ref{prop:PSa} into our next lemma (cf.  \cite[Corollary 2]{PS88}).

\begin{lem}\label{lem:PSlem} Let $\delta > 0$, and let $\mathcal{A}$ be a set of integers with upper logarithmic density exceeding $\delta$. Then there are numbers $a< a'$ in $\mathcal{A}$ for which $a\mid a'$ and $a'/a \ll_{\delta} 1$.
\end{lem}

\begin{proof} Let $\Pp$ consist of the first $k$ primes, where $k=k(\delta)$ is chosen minimally to make $\sum_{p \in \Pp}\frac{1}{p} > c_1$ and $10(\sum_{p \in \Pp} \frac{1}{p})^{-1/2} < \delta$. Since $\mathcal{A}$ has upper logarithmic density $> \delta$, there are arbitrarily large $N$ for which
\[ \sum_{\substack{a \in \A \\ a\le N}} \frac{1}{a} > 10 (\log{N}) \left(\sum_{p \in \Pp}\frac{1}{p}\right)^{-1/2}. \] Proposition \ref{prop:PSa} now yields the existence of $a<a'$ in $\mathcal{A}$ where $a \mid a'$ and $\frac{a'}{a} \le \prod_{p \in \Pp}p \ll_{\delta} 1$.
\end{proof}

We also need a classical result of Bang on primitive prime divisors of Mersenne numbers (\cite{bang86}; see also \cite{roitman97}).

\begin{lem} For each integer $m>1$, $m \ne 6$, there is a prime $p$ for which $\ell(p)=m$.
\end{lem}

We can now prove that the distribution function of $f(n)$ is continuous.

\begin{proof}[Proof of Theorem \ref{thm:main2}, for $f(n)$]

 Let $D(u)$ be the distribution function of $f$, and suppose  for a contradiction that $D$ is not continuous. Let $\alpha$  be a discontinuity of $D$ (necessarily a jump discontinuity) and let $\delta = D(\alpha+)-D(\alpha-)$. Then $\delta > 0$ and for every $\epsilon > 0$,
\begin{equation*} D(\alpha+\epsilon)-D(\alpha-\epsilon) \ge \delta. \end{equation*}
Except possibly on a countable set of $\epsilon$ (specifically, those $\epsilon$ where one of $\alpha\pm \epsilon$ is a discontinuity of $D$), the left-hand difference can be interpreted as the density of $a$ with $f(a) \in (\alpha-\epsilon,\alpha+\epsilon)$. In particular, there are arbitrarily small positive $\epsilon$ where that density exists and is at least $\delta$. For these $\epsilon$, the corresponding logarithmic density also exists and is at least $\delta$.

The set of $a$ with squarefull part exceeding $K$ has upper logarithmic density bounded by $\sum_{\text{squarefull }m \ge K} \frac{1}{m}$, which is smaller than $\frac{1}{2}\delta$ if $K$ is fixed sufficiently large in terms of $\delta$. Putting this together with the conclusion of the last paragraph, we conclude that there are arbitrarily small positive $\epsilon$ for which
\[ \A = \{a: \text{$a$ has squarefull part at most $K$, and } f(a) \in (\alpha-\epsilon,\alpha+\epsilon)\} \]
has lower logarithmic density exceeding $\frac{1}{2}\delta$.  By Lemma \ref{lem:PSlem}, we may select $a < a'$ in $\mathcal{A}$ for which $a \mid a'$ and $\frac{a'}{a} \ll_{\delta} 1$. Take $q$ to be a prime appearing in $a'$ to a higher power than in $a$, and say that $q^e\parallel a'$. If $e > 1$, then $q^e \le K$, while if $e = 1$, then $q^e = q \mid a'/a$. Either way, $$q^e \ll_{\delta} 1.$$ Now notice that
\begin{equation*}
f(a') - f(a) \ge \sum_{\ell(p) = q^e} \frac{1}{p}. \end{equation*}
By Bang's  theorem, there is at least one prime $p$ with $\ell(p) = q^e$, so the right-hand sum is nonempty. Moreover, if $p$ is any such prime, $p < 2^{q^e} \ll_{\delta} 1$, and so
\[ f(a') -f(a) \gg_{\delta} 1. \]
But $f(a), f(a') \in (\alpha-\epsilon,\alpha+\epsilon)$,  so that
\[ f(a') - f(a) < 2\epsilon. \]
The last two displays are incompatible for sufficiently small values of $\epsilon>0$.
\end{proof}

A bit of staring reveals that we have actually shown a more general result.

\begin{thm}\label{thm:genmain2} Suppose $a(n), b(n)$ are nonnegative-valued arithmetic functions connected by the relation
\begin{equation}\label{eq:abdef} a(n) = \sum_{d\mid n} b(d) \qquad\text{for all positive integers $n$}. \end{equation}
Suppose that $b(q^e)>0$ for every prime power $q^e$ {\rm (}including the case $e=1${\rm )}. Then for each real $\alpha$ and every $\delta>0$, there is an $\epsilon>0$ such that the upper logarithmic density of $n$ with $a(n) \in (\alpha-\epsilon,\alpha+\epsilon)$ is smaller than $\delta$. In particular, if $a(n)$ has a distribution function, that distribution function is continuous.
\end{thm}

\noindent That the distribution function of $f(n)$ is continuous follows from choosing $a(n)=f(n)$ and $b(n) = \sum_{\ell(p)=n} 1/p$ in this general result. The second assertion of Theorem \ref{thm:genmain2}, that the distribution function of $F(n)$ is continuous, follows from choosing $a(n) = F(n)$ and $b(n) = E(n)$. (Note that Bang's theorem implies that $E(q^e)>0$ for all prime powers $q^e$.)

Several variants of Theorem \ref{thm:genmain2} are possible. For instance, instead of requiring that $b(q^e)>0$ for all prime powers $q^e$, it is enough to impose this condition for the prime powers $q^e$ with $q > q_0$, for any fixed $q_0$. To see this, observe that the proof of Lemma \ref{lem:PSlem} can be carried out including only primes larger than $q_0$ into $\Pp$. This allows one to impose the additional constraint in the conclusion of that lemma that the quotient $a'/a$ have no prime factors up to $q_0$. In the proof of Theorem \ref{thm:genmain2}, $q$ will be  chosen as a divisor of $a'/a$, and so the prime power $q^e$ will have $q>q_0$.

\begin{examples}\mbox{ }
\begin{enumerate}
\item[(a)] Let $a(n) = \sum_{p\mid F_n} 1/p$, where $F_n$ is the $n$th Fibonacci number. It follows from the general results of \cite{LS07} that $a(n)$ has a distribution function. To show continuity, we apply the variant of Theorem \ref{thm:genmain2} just discussed. Let $z(p)$ denote the smallest natural number $n$ for which $p\mid F_n$ (the \textsf{rank of appearance of $p$}). Then \eqref{eq:abdef} holds with $b(n) = \sum_{z(p)\mid n} 1/p$. That $b(q^e)>0$ for all odd prime powers $q^e$ is a consequence of Carmichael's 1913 theorem \cite{carmichael13} that $z(p)=m$ has a prime solution $p$ for each $m\notin \{1,2,6,12\}$. So our variant of Theorem \ref{thm:genmain2} applies, and the distribution function of $a(n)$ is continuous.
\item[(b)] In view of results such as Lemma \ref{lem:GRHlem}, it is natural to compare $f(n)$ with  $f^{*}(n) := \sum_{p-1\mid n}1/p$. Theorem \ref{thm:genmain2} does not apply in this example, and, unlike in (a), there is no obvious variant that does apply. There turns out to be a very good reason for these difficulties: $f^{*}(n)$ does not have a continuous distribution function. Indeed, is it clear that $f^{*}(n) = \frac{1}{2}$ precisely on the density $\frac{1}{2}$ set of odd numbers. With somewhat more effort, the ideas of \cite{PS88} can be pushed to show that $f^{*}(n)$ has a discrete distribution function, with jumps precisely at the elements of $f^{*}(\mathbb{N})$.
\end{enumerate}
\end{examples}

\section{Concluding remarks on generalizations} In all of the theorems appearing in the introduction, $2^n-1$ can be replaced with $a^n-1$ for any fixed integer $a>1$. In the statement of Theorem \ref{thm:main1}, one should then replace $\frac{1}{2}$ with $\sum_{p\mid a}\frac{1}{p}$, and in Theorem \ref{thm:main1point5}, $\log{2}$ should be replaced with $\log \frac{a}{\phi(a)}$. In Theorem \ref{thm:1andthreequarters}, $\frac{1}{2}$ should be replaced with $\frac{\phi(a)}{a}$. Only superficial changes to the proofs are required. Theorem \ref{thm:main2} remains true as stated, by more or less the same proof (we use that $\ell_a(p)=n$ has at least one solution $p$ for each $n>6$; see \cite{roitman97}).

A natural domain of generalization is provided by the nondegenerate Lucas sequences of the form $U_n(P,Q)$, where $P$ and $Q$ are coprime integers. (See Chapter 1 of \cite{ribenboim00} for a refresher on the basic properties of these sequences.) If $a$ is an integer larger than $1$, and we take $P=a+1$ and $Q=a$, then $U_n(P,Q) = \frac{a^{n}-1}{a-1}$. For our purposes, $\frac{a^n-1}{a-1}$ is more or less the same as $a^n-1$, and so this setting morally subsumes the one considered in the last paragraph.

For these sequences $U_n$, a positive integer $m$ divides some term $U_n$ precisely when $m$ is coprime to $Q$. In that case, the set of corresponding $n$ consists of all multiples of the rank of appearance $z(m)$, which is analogous to $\ell(m)$. If we set $\Delta = P^2-4Q$, then $z(p)\mid p-\leg{\Delta}{p}$ for each odd prime $p\nmid PQ\Delta$. Furthermore, for each prime $p$ not dividing $Q$, and every positive integer $k$, we have $z(p^k) \mid z(p)p^{k-1}$. Also, $z(m) = \lcm_{p^k\parallel m} z(p^k)$.

We expect $U_n$-analogues of all of our theorems. For Theorem \ref{thm:main2}, these analogues are fairly simple to show (and the case $P=1, Q=-1$, corresponding to the Fibonacci numbers, was already discussed): Again \cite{LS07} gives existence of the distribution function. It is known that for each $n>30$, there is a prime $p$ with $z(p) = n$ \cite{BHV01}. This allows us to deduce continuity from the variant of Theorem \ref{thm:genmain2} discussed above.

The $U_n$ versions of Theorems \ref{thm:main1}, \ref{thm:main1point5}, and \ref{thm:1andthreequarters} also appear attainable. Their statements should be modified as in the first paragraph of this remark, but with $|Q|$ replacing $a$. However, the proofs cannot be viewed as routine, as the literature does not seem to contain explicit analogues of Proposition \ref{prop:granville},  Lemma \ref{lem:GRHlem}, Proposition \ref{prop:FPS} (see p.\ 13 of \cite{ribenboim00} for the $U_n$-analogue of $\lambda$), or Proposition \ref{prop:kurlberg}. However, we believe all of these to be within reach. To give one example, it should be possible to prove the needed variant of Lemma \ref{lem:GRHlem} following Kurlberg and Pomerance, starting from Chen's arguments in \cite{chen02}.

One can also consider analogues of $f$ and $F$ associated to the ``companion'' $V_n$ Lucas sequences. Extending Erd\H{o}s's upper bound \eqref{eq:erdosfupper} to $V_n$ requires slightly more arithmetic. Whereas all but finitely many primes divide some term in the $U_n$ sequence, with finitely many exceptions a prime $p$ divides some term in the $V_n$ sequence if and only if $z(p)$ is defined and even. Using results of Moree \cite{moree97, moree98}, the first author \cite{engberg14} showed that
\[\max_{n \le x} \sum_{p \mid a^n + b^n} \frac1p = \frac13 \log\log\log x + O_{a, b}(1)\]
for $a, b$ positive integers with $a / b$ not a rational square. For an arbitrary nondegenerate Lucas sequence, the constant $1/3$  should be replaced by some density $\delta\in (0, 1)$ depending on $a$ and $b$. We believe that Theorems \ref{thm:main1}, \ref{thm:main1point5}, and \ref{thm:1andthreequarters}, can be extended to $V_n$ Lucas sequences as well.

Establishing the analogue of Theorem \ref{thm:main2} for $V_n$ is easier. While $V_n$ is not a divisibility sequence, it is still the case that $V_n \mid V_{mn}$ whenever $m$ is odd, and this is enough to allow a modified of the argument to run. (To get an idea of how this goes, suppose $a,a'$ are positive integers with $a'/a$ odd, and where $q^e$ is an odd prime power dividing $a'$ but not $a$. If $z(p)=2^{r+1} q^e$, where $2^r \parallel a$, then $p \mid V_{a'}$ but $p\nmid V_{a}$. To make this observation useful, we use again that $z(p)=n$ is solvable for all large $n$, along with the fact that most integers are not divisible by a large power of $2$.)

\subsection*{Acknowledgements}
The authors thank Carl Pomerance and the anonymous referee for several helpful suggestions. The second author is supported by NSF Award DMS-2001581.

\providecommand{\bysame}{\leavevmode\hbox to3em{\hrulefill}\thinspace}
\providecommand{\MR}{\relax\ifhmode\unskip\space\fi MR }
\providecommand{\MRhref}[2]{%
  \href{http://www.ams.org/mathscinet-getitem?mr=#1}{#2}
}
\providecommand{\href}[2]{#2}

\end{document}